\documentclass[11pt]{article}%
\usepackage{amssymb,amsmath,amsfonts,amsthm,array,bm,color}%
\usepackage{amsmath}%
\usepackage{amsfonts}%
\usepackage{amssymb}%
\usepackage{graphicx}
\providecommand{\U}[1]{\protect\rule{.1in}{.1in}}

\numberwithin{equation}{section}

\newtheorem{theorem}{Theorem}[section]
\newtheorem{lemma}[theorem]{Lemma}
\newtheorem{corollary}[theorem]{Corollary}
\newtheorem{proposition}[theorem]{Proposition}
\newtheorem{remark}[theorem]{Remark}
\newtheorem{example}[theorem]{Example}

\def\<{\langle}
\def\>{\rangle}
\def\d{{\rm d}}

\def\div{{\rm div}}
\def\E{\mathbb{E}}

\def\N{\mathbb{N}}
\def\P{\mathbb{P}}
\def\R{\mathbb{R}}
\def\T{\mathbb{T}}
\def\Z{\mathbb{Z}}

\def\eps{\varepsilon}

\begin{document}


\title{Regularization by noise for the point vortex model of\\ mSQG equations}

\author{Dejun Luo\footnote{Email: luodj@amss.ac.cn. RCSDS, Academy of Mathematics and Systems Science, Chinese Academy of Sciences, Beijing 100190, China, and School of Mathematical Sciences, University of the Chinese Academy of Sciences, Beijing 100049, China. } \quad Martin Saal\footnote{Email: msaal@mathematik.tu-darmstadt.de. Department of Mathematics, TU Darmstadt, Schlossgartenstr. 7, 64289 Darmstadt, Germany }}

\maketitle

\begin{abstract}
We consider the point vortex model corresponding to the modified Surface Quasi-Geostrophic (mSQG) equations on the two dimensional torus. It is known that this model is well posed for almost every initial conditions. We show that, when the system is perturbed by a certain space-dependent noise, it admits a unique global solution for any initial configuration. We also present an explicit example for the deterministic system where three different point vortices collapse.
\end{abstract}

\textbf{Keywords:} Point vortices, modified Surface Quasi-Geostrophic equations, space-dependent noise, absolute continuity


\section{Introduction}

Let $\T^2=\R^2/\Z^2$ be the 2D torus. We will think of $\T^2$ as $[-1/2,1/2]^2$ equipped with the periodic boundary condition. Recall the differential operators $\nabla^\perp=(\partial_{x_2}, -\partial_{x_1})$ and $\Delta= \partial_{x_1}^2 + \partial_{x_2}^2$. Consider the Surface Quasi-Geostrophic (SQG) equation on $\T^2$:
  \begin{equation*}
  \left\{ \aligned
  & \partial_t \theta+ u\cdot \nabla\theta =0, \\
  & u= \nabla^\perp (-\Delta)^{-1/2} \theta,
  \endaligned
  \right.
  \end{equation*}
which is widely used in meteorological and oceanic flows to describe the temperature $\theta$ in a rapidly rotating stratified fluid with uniform potential vorticity, see e.g. \cite{HPGS, SBHMT} for the geophysical background. It is known (see \cite{CMT}) that the above equation in 2D has some structural similarities with the 3D Euler equations, for this reason it attracted a lot of attention in the mathematics community.

We are interested in the following modified Surface Quasi-Geostrophic (mSQG) equation:
  \begin{equation}\label{mSQG}
  \left\{ \aligned
  & \partial_t \theta+ u\cdot \nabla\theta =0, \\
  & u= \nabla^\perp (-\Delta)^{-(1+\eps)/2} \theta,
  \endaligned
  \right.
  \end{equation}
where $\eps\in [0,1]$. If $\eps=0$, then the above equation reduces to the SQG equation, while if $\eps=1$, then \eqref{mSQG} gives rise to the vorticity form of 2D Euler equation. Thus, the mSQG equation serves as a bridge between the SQG equation and the 2D Euler equation. This family of equations was introduced in \cite{CCW, CFMR} to approach the SQG equation by smoother models. We refer the readers to the introduction of \cite{FS} for a detailed list of well posedness results on the equations \eqref{mSQG}.

We will study the model of $N$ point vortices corresponding to \eqref{mSQG}:
  \begin{equation}\label{vortex-model}
  \frac{\d x_i(t)}{\d t}= \sum_{j=1,j\neq i}^N \xi_j K_\eps (x_i(t) -x_j(t)),\quad i=1,\cdots, N,
  \end{equation}
where $\xi_j\in \R\setminus \{0\}$ is the intensity of the vortex point $x_j(t)$, $j=1,\cdots, N$, and
\begin{align*}
 K_\eps: \T^2\setminus\{0\}\to \R^2, \quad  K_\eps(x)=c_\eps \sum_{n\in\Z^2} \frac{(x+n)^\perp}{\left|x+n\right|^{3-\eps}}
\end{align*}
for some $c_\eps>0$ is the kernel associated to the operator $\nabla^\perp (-\Delta)^{-(1+\eps)/2}$. When $\eps=1$, $K_1$ is the well known Biot--Savart kernel on $\T^2$; the vortex dynamics \eqref{vortex-model} on the plane was treated systematically by Marchioro and Pulvirenti in \cite[Chap. 4]{MP}. In particular, there exist examples of initial configurations (see \cite[Section 4.2]{MP}) starting from which the vortex system collapses in finite time, i.e., initially distinct vortex points meet each other. Nevertheless, it can be shown that, for ${\rm Leb}_{(\T^2)^N}$-a.e. initial configuration in $(\T^2)^N$, the system \eqref{vortex-model} of equations has a global solution, see \cite{DP}, \cite[Section 4.4]{MP} or \cite[Appendix]{F1}. The readers can find in \cite[Section III]{BB} some discussions on possible collapse of point vortex systems corresponding to the mSQG equation, i.e. $\eps\in (0,1)$ in \eqref{vortex-model}. Following the ideas of Marchioro and Pulvirenti  \cite[Section 4.6]{MP}, we also provide in the last section some explicit conditions for collapse. The almost everywhere well posedness of the system \eqref{vortex-model} has been proved in \cite{CGM, FS, GR}.

It is clear that one can improve the well posedness of the deterministic system \eqref{vortex-model} by perturbing vortex points with mutually independent Brownian motions; however, such stochastic system does not correspond to the Lagrangian formulation of the stochastic mSQG equation. On the other hand, Flandoli et al. \cite{FGP} proved that, when perturbed by a certain space-dependent non-degenerate noise, the stochastic vortex model of the 2D Euler equation is fully well posed for every initial configuration. This is a typical example of the phenomenon of regularization by noise. Our purpose is to show that similar noises also restore well posedness of the vortex model \eqref{vortex-model} of mSQG equation. To this end, we perturb the system by a space-dependent random noise:
  \begin{equation}\label{stoch-vortex-model}
  \d x_i(t) =\sum_{j=1,j\neq i}^N \xi_j K_\eps(x_i(t) -x_j(t))\,\d t + \sum_{l=1}^\infty \sigma_l(x_i(t))\circ \d W^l_t, \quad i=1,\cdots, N,
  \end{equation}
where $\{W^l_t\}_{l\geq 1}$ is a sequence of independent standard Brownian motions defined on some probability space $(\Omega, \mathcal F, \P)$, and $\{\sigma_l\}_{l\geq 1}$ a family of smooth divergence free vector fields on $\T^2$. Suppose that the above stochastic system is globally well posed and define
  $$\theta_t = \sum_{i=1}^N \xi_i \delta_{x_i(t)},\quad t>0.$$
Then, heuristically, one can show (cf. \cite[Section 2.3]{FGP}) that $\theta_t$ satisfies the stochastic mSQG equation: for any $\phi\in C^\infty(\T^2)$,
  \begin{equation}\label{stoch-mSQG}
  \d \<\theta_t, \phi\>= \<\theta_t, u_t\cdot \nabla \phi\> \,\d t + \sum_{l=1}^\infty \<\theta_t, \sigma_l\cdot \nabla \phi\>\circ \d W^l_t,
  \end{equation}
provided that the vector field
  $$u_t(x) = \<\theta_t, K_\eps(x-\cdot)\>= \sum_{i=1}^N \xi_i K_\eps(x- x_i(t))$$
is properly interpreted at the vortex points $x_i(t),\, 1\leq i\leq N$. Indeed, \eqref{stoch-mSQG} makes sense if
  $$u_t(x_i(t))= \sum_{j=1, j\neq i}^N \xi_j K_\eps(x_i(t)- x_j(t)),$$
which is in accordance with \eqref{vortex-model}.

To state our assumptions, we introduce the vector fields on $(\T^2)^N$:
  $$A_l(X)= (\sigma_l(x_1), \ldots, \sigma_l(x_N)),\quad X=(x_1,\ldots, x_N)\in (\T^2)^N,\ l\geq 1.$$
Denote by $\Delta_N$ the generalized diagonal of $(\T^2)^N$, i.e.
  $$\Delta_N= \big\{ X=(x_1,\ldots, x_N)\in (\T^2)^N: \exists \ i\neq j \mbox{ such that } x_i=x_j \big\}.$$
Here are the main assumptions on $\{\sigma_l\}_{l\geq 1}$ (see Section \ref{sect-example} for an example):

\begin{itemize}
\item[\textbf{(H1)}] The vector fields $\sigma_l$ are periodic, smooth and $\div(\sigma_l)=0$ for all $l\geq 1$.
\item[\textbf{(H2)}] (Ellipticity) The vector space spanned by the vectors $\{A_l(X)\}_{l\geq 1}$ is the whole $\R^{2N}$ for every $X\in \Delta_N^c = (\T^2)^N \setminus \Delta_N$.
\end{itemize}

We give some remarks on the above hypotheses.

\begin{remark}\label{1-rem}
\begin{itemize}
\item[\rm (a)] By multiplying each $\sigma_l$ with a positive constant $c_l$ which decreases fast enough to 0 (cf. Example \ref{exam-1}), we can assume that the SDE
  $$\d y(t) = \sum_{l=1}^\infty \sigma_l(y(t))\circ \d W^l_t, \quad y(0)\in \T^2 $$
generates a stochastic flow of $C^1$-diffeomorphisms on $\T^2$. Indeed, it is easy to show that $\{c_l \sigma_l\}_{l\geq 1}$ still satisfy the hypotheses {\rm \textbf{(H1)}} and {\rm \textbf{(H2)}}. In particular, there is a constant $C>0$ such that
  $$\sum_{l=1}^\infty c_l^2 |\sigma_l(x)-\sigma_l(y)|^2 \leq C|x-y|^2 \quad \mbox{for all } x,y\in \T^2,$$
which holds with $C= \sum_{l= 1}^\infty c_l^2 \|\nabla \sigma_l \|_\infty^2 <\infty $.

\item[\rm (b)] We can even assume that $\sum_{l=1}^\infty \sigma_l(x)\cdot \nabla \sigma_l(x)=0$ for all $x\in \T^2$. In this case, the Stratonovich equation \eqref{stoch-vortex-model} has the same form with the It\^o equation. To show this assumption, note that in the space-homogeneous case (see the discussion below \cite[Lemma 3.2]{DFV}), we have for all $\alpha,\beta =1,2$,
      $$\sum_{l=1}^\infty \sigma_l^\alpha(x) \sigma_l^\beta(x) = Q^{\alpha \beta}\quad \mbox{for all } x\in \T^2,$$
    where $Q^{\alpha \beta}$ is a constant. Therefore,
      $$\sum_{l=1}^\infty \bigg[ \bigg(\frac{\partial}{\partial x_\alpha} \sigma_l^\alpha(x)\bigg) \sigma_l^\beta(x) + \sigma_l^\alpha(x) \frac{\partial}{\partial x_\alpha} \sigma_l^\beta(x) \bigg]=0. $$
    Summing over $\alpha=1,2$ and using  $\div (\sigma_j)=0$, we obtain $\sum_{l=1}^\infty \sigma_l(x)\cdot \nabla \sigma_l^\beta (x)=0$.
\end{itemize}
\end{remark}

Now we can state our main result; recall that the result proved in \cite{FGP} corresponds to the case $\eps =1$.

\begin{theorem}\label{main-result}
Fix any $\eps\in(0,1)$. Under the hypotheses {\rm \textbf{(H1)}} and {\rm \textbf{(H2)}}, the stochastic point vortex system \eqref{stoch-vortex-model} has a unique global solution for any initial data $X\in \Delta_N^c$.
\end{theorem}

This paper is organized as follows. In Section 2 we provide some preliminary results concerning the regularity of the kernel $K_\eps$, and an explicit example of vector fields satisfying the hypotheses. We also recall a result on the existence of densities for solutions to stochastic differential equations. Theorem \ref{main-result} will be proved in Section 3, mainly following the idea of \cite{FGP}. Finally, in Section 4, we carry out some detailed computations which lead to explicit conditions for possible collapse of the deterministic vortex system \eqref{vortex-model}.

\section{Some preparations}

In this section we make some preparations by recalling regularity properties of the kernel $K_\eps$ and the noise that we will use to regularize the vortex system \eqref{vortex-model}. A result on the existence of densities for solutions to SDEs is recalled in Section \ref{subsec-density} for later use.

\subsection{Regularity of the kernel $K_\eps$} \label{sect-kernel}

The singular interaction kernel $K_\eps$ in \eqref{vortex-model} is locally of the form
  $$K_\eps(x)\sim \frac{x^\perp}{|x|^{3-\eps}}, \quad \eps\in (0,1).$$
Indeed, $K_\eps=\nabla^\perp G_\eps = (\partial_{x_2} G_\eps, - \partial_{x_1} G_\eps)$, where
\begin{align*}
 G_\eps: \T^2\setminus\{0\}\to \R, \quad  G_\eps(x)=c_\eps\sum_{n\in\Z^2} \frac{1}{\left|x+n\right|^{1-\eps}}
\end{align*}
is the Green function associated to the operator $\nabla^\perp (-\Delta)^{-(1+\eps)/2}$; up to a multiplicative constant, one has
  $$G_\eps(x) \sim - \frac{1}{|x|^{1-\eps}} \quad \mbox{as } |x|\to 0.$$

Since the function $G_\eps$ is singular near the origin, for any $\delta\in (0,1)$, we consider a smooth periodic function $G_\eps^\delta:\T^2 \to \R$ satisfying
  $$\aligned
  G_\eps^\delta(x) & = G_\eps(x) \quad \mbox{for } |x|\geq \delta,
  \endaligned$$
and for $i=0,1,2$,
  \begin{equation}\label{property-regular-Green}
  \big|\nabla^i G_\eps^\delta(x) \big|  \leq \frac{C}{|x|^{i+1-\eps}}
  \end{equation}
for all $|x|>0$ and some constant $C>0$. In the next section we shall make use of the regular kernel
  \begin{equation}\label{regular-kernel}
  K_\eps^\delta(x) = \nabla^\perp G_\eps^\delta(x), \quad x\in \T^2,
  \end{equation}
which is a smooth and divergence free vector field on $\T^2$.

\subsection{Vector fields verifying the hypotheses \textbf{(H1)} and \textbf{(H2)}} \label{sect-example}

The readers are referred to \cite[Section 3]{DFV} for general discussions on the examples of vector fields satisfying the conditions \textbf{(H1)} and \textbf{(H2)}. In our special case of 2D torus $\T^2$, we can present a more explicit example of vector fields $\{\sigma_l\}_{l\geq 1}$ with the above-mentioned properties.

\begin{example}\label{exam-1}
Let $\{e_k\}_{k\in \Z_0^2}$ be defined as
  \begin{equation*}\label{real-basis}
  e_k(x) = \sqrt{2} \begin{cases}
  \cos(2\pi k\cdot x), & k\in \Z^2_+ , \\
  \sin(2\pi k\cdot x), & k\in \Z^2_-,
  \end{cases} \quad  x\in \T^2,
  \end{equation*}
where $\Z_0^2= \Z^2 \setminus \{0\}$, $\Z^2_+ = \big\{k\in \Z^2_0: (k_1 >0) \mbox{ or } (k_1=0,\, k_2>0) \big\}$ and $\Z^2_- = -\Z^2_+$. This family of functions is an orthonormal basis of square integrable functions on $\T^2$ with vanishing mean. Define
  \begin{equation}
  \sigma_k = \frac{k^\perp}{|k|^\gamma} e_k, \quad k\in \Z_0^2,
  \end{equation}
where $\gamma>3$ is a constant. It is obvious that the family $\{\sigma_k\}_{k\in \Z_0^2}$ of vector fields fulfill {\rm \textbf{(H1)}}. Moreover, one can also easily check that they satisfy the properties discussed in (a) and (b) of Remark \ref{1-rem}.

Next we show that $\{\sigma_k\}_{k\in \Z_0^2}$ also verifies {\rm \textbf{(H2)}}. For any $X= (x_1,\ldots, x_N)\in \Delta_N^c$ and $V= (v_1, \ldots, v_N)\in \R^{2N}$,
  $$\sum_{k\in \Z_0^2} \<A_k(X),V\>_{\R^{2N}}^2 = \sum_{k\in \Z_0^2} \frac1{|k|^{2\gamma}} \Bigg(\sum_{\alpha=1}^N (k^\perp \cdot v_\alpha) e_k(x_\alpha) \Bigg)^2.$$
It is enough to show that, if the above quantity vanishes, then one must have $v_i=0$ for all $i=1,2, \ldots, N$. The arguments are analogous to those of \cite[Remark 3.3]{DFV}. If the above quantity vanishes, then for any $k\in \Z^2_0$, we have $\sum_{\alpha=1}^N e_k( x_\alpha) ( k^\perp\cdot v_\alpha ) =0$. Equivalently,
  $$\sum_{\alpha=1}^N e_k( x_\alpha) \big( k\cdot v_\alpha^\perp \big) =0 \quad \mbox{for all } k\in \Z_0^2. $$
Take a smooth real valued function $\varphi$ on $\T^2$ with zero mean, then
  $$0= \sum_{\alpha=1}^N \sum_{k\in \Z^2} \hat\varphi(k) \, e_k( x_\alpha) \big( k\cdot v_\alpha^\perp \big) = \frac{1}{2\pi} \sum_{\alpha=1}^N \big(\nabla \varphi(x_\alpha) \cdot v_\alpha^\perp \big) . $$
Since $x_1,\ldots, x_N$ are mutually distinct, we can construct a function $\varphi$ such that $\nabla \varphi(x_\alpha) = v_\alpha^\perp$. This implies $v_\alpha =0$ for all $\alpha=1, \ldots, N$.
\end{example}

\subsection{Existence of densities for solutions of SDEs} \label{subsec-density}

Here we consider the SDE on $\R^d$ with infinitely many noises:
  \begin{equation}\label{sde}
  \d X_t = b(X_t)\,\d t + \sum_{l=1}^\infty A_l(X_t)\,\d W^l_t,\quad X_0 = x_0\in \R^d,
  \end{equation}
where $b:\R^d \to \R^d$ and $A_l:\R^d \to \R^d$ are globally Lipschitz continuous vector fields, and
  $$\sum_{l=1}^\infty |A_l(x) - A_l(y)|^2 \leq C|x-y|^2, \quad x,y\in \R^d$$
for some $C>0$. Assume also
  $$\sum_{l=1}^\infty |A_l(x)|^2 \leq C(1+|x|^2), \quad x\in \R^d, $$
thus the covariance matrix
  $$Q(x)= \sum_{l=1}^\infty A_l(x)\otimes A_l(x)$$
is well defined. The next result is due to Bouleau and Hirsh \cite{Bouleau}, see also \cite[Theorem 2.3.1]{Nualart}.

\begin{theorem}\label{thm-density}
Assume that for any $t>0$, one has
  \begin{equation}\label{thm-density-1}
  \P\bigg(\int_0^t {\bf 1}_{\{ {\rm det}\, Q(X_s) \neq 0 \}} \,\d s >0 \bigg) =1,
  \end{equation}
then the law of $X_t$ has a density with respect to the Lebesgue measure for all $t>0$.
\end{theorem}

A simple sufficient condition for \eqref{thm-density-1} is when $Q(x_0)$ is nondegenerate, i.e., the vectors $A_l(x_0),\, l\geq 1$ span the whole space $\R^d$.

\begin{corollary}\label{cor-density}
If ${\rm det}\, Q(x_0)>0$, then for any $t>0$, the law of $X_t$ is absolutely continuous with respect to the Lebesgue measure.
\end{corollary}

Indeed, by the continuity of $Q(x)$, we can find a neighborhood $U(x_0)$ of $x_0$ such that ${\rm det}\, Q(x)>0$ for all $x\in U(x_0)$. The above result follows easily from Theorem \ref{thm-density} and the continuity of paths of the solution $X_t$ to \eqref{sde}, see also \cite[Corollary 18]{FGP} for the case of finite noises and \cite[Proposition 2.3]{NNS} for the case of infinite noises.

\section{Proof of the main result}

Since the parameter $\eps\in(0,1)$ is fixed, we shall omit it throughout this section. Recall the regular kernel $K^\delta$ defined in Section \ref{sect-kernel}. We write $X^\delta_t = \big(x^\delta_1(t),\ldots, x^\delta_N(t) \big)$ for the unique solution of
  \begin{equation}\label{regular-vortex-model}
  \d x^\delta_i(t) =\sum_{j=1,j\neq i}^N \xi_j K^\delta \big(x^\delta_i(t) -x^\delta_j(t) \big)\,\d t + \sum_{l=1}^\infty \sigma_l \big(x^\delta_i(t) \big)\circ \d W^l_t, \quad i=1,\cdots, N,
  \end{equation}
with initial condition $X_0= (x_1(0), \ldots, x_N(0))\in (\T^2)^N$. Indeed, the above system of SDEs generate a stochastic flow $\big\{X^\delta_t \big\}_{t\geq 0}$ of diffeomorphisms on $(\T^2)^N$, see \cite[Section 4.7]{Kunita}. By Remark \ref{1-rem}(b), the equations can be equally written in the It\^o form.

Noticing that the vector fields $K^\delta$ and $\sigma_l$ are divergence free on $\T^2$, we have the following simple result.

\begin{lemma}\label{3-lem-1}
A.s., for any $t>0$, the mapping $X^\delta_t: (\T^2)^N \ni X_0 \mapsto X^\delta_t(X_0)\in (\T^2)^N$ preserves the Lebesgue measure on $(\T^2)^N$: for any integrable function $f:(\T^2)^N \to \R$,
  $$\int_{(\T^2)^N} f\big( X^\delta_t(X_0) \big)\,\d X_0 = \int_{(\T^2)^N} f(Y)\,\d Y.$$
\end{lemma}

Let $g^\delta: (\T^2)^N \to \R$ be an auxiliary function defined as
  $$g^\delta(X)= -\sum_{1\leq i\neq j \leq N} \big( G^\delta(x_i-x_j) -c_0\big), \quad X=(x_1,\ldots, x_N)\in (\T^2)^N,$$
where $c_0>0$ is a constant independent of $\delta$ such that $G^\delta(x) \leq c_0,\, x\in \T^2$. We prove the following key estimate.

\begin{proposition}\label{3-prop-2}
Let $X^\delta_t$ be the flow on $(\T^2)^N$ associated to \eqref{regular-vortex-model}. Then there are constants $C_1, C_2>0$ such that
  $$\E\bigg[ \sup_{t\in [0,T]} g^\delta\big( X^\delta_t(X_0) \big) \bigg] \leq g^\delta(X_0) + C_1 \int_0^T \E\, h_1\big( X^\delta_t(X_0) \big) \,\d t + C_2 \bigg[\int_0^T \E\, h_2\big( X^\delta_t(X_0) \big) \,\d t \bigg]^{1/2}, $$
where $h_1, h_2: (\T^2)^N \to \R_+$ are two integrable functions defined as
  $$\aligned
  h_1(X) &= \sum_{i,j,k=1,\, i\neq j, i\neq k, j\neq k}^N \frac1{| x_i -x_j|^{2-\eps} | x_i -x_k |^{2-\eps}} + \sum_{1\leq i\neq j \leq N} \frac1{| x_i -x_j |^{1-\eps}}, \\
  h_2(X) &= \sum_{1\leq i\neq j\leq N} \frac1{| x_i -x_j|^{2-2\eps}}, \qquad X=(x_1, \ldots, x_N)\in (\T^2)^N.
  \endaligned$$
\end{proposition}

\begin{proof}
We follow the idea of the proof of \cite[Lemma 4]{FGP}, see also \cite[Lemma 9]{FS}. In the following we write $x^\delta_i(t)$ for the $i$-th component of $X^\delta_t(X_0)$, $i=1,\ldots, N$. By the It\^o formula,
  \begin{equation} \label{3-prop-2.1}
  \aligned
  \d G^\delta \big(x^\delta_i(t)- x^\delta_j(t)\big)=&\, \nabla G^\delta \big(x^\delta_i(t)- x^\delta_j(t)\big) \cdot \d \big(x^\delta_i(t)- x^\delta_j(t)\big) \\
  & + \frac12 \sum_{\alpha,\beta=1}^2 \partial_{\alpha,\beta} G^\delta \big(x^\delta_i(t)- x^\delta_j(t)\big)\, \d \Big[ \big(x^\delta_i- x^\delta_j\big)_\alpha, \big(x^\delta_i- x^\delta_j\big)_\beta \Big]_t\, .
  \endaligned
  \end{equation}
Using the equation \eqref{regular-vortex-model}, we have
  $$\aligned
  \d \big(x^\delta_i(t)- x^\delta_j(t)\big) =&\, \sum_{k\neq i} \xi_k K^\delta \big(x^\delta_i(t) -x^\delta_k(t) \big)\,\d t - \sum_{k\neq j} \xi_k K^\delta \big(x^\delta_j(t) -x^\delta_k(t) \big)\,\d t \\
  &+ \sum_{l=1}^\infty \big[\sigma_l \big(x^\delta_i(t) \big) - \sigma_l \big(x^\delta_j(t) \big) \big]\, \d W^l_t
  \endaligned $$
and hence
  $$\d \Big[ \big(x^\delta_i- x^\delta_j\big)_\alpha, \big(x^\delta_i- x^\delta_j\big)_\beta \Big]_t = \sum_{l=1}^\infty \big[\sigma_l^\alpha \big(x^\delta_i(t) \big) - \sigma_l^\alpha \big(x^\delta_j(t) \big) \big] \big[\sigma_l^\beta \big(x^\delta_i(t) \big) - \sigma_l^\beta \big(x^\delta_j(t) \big) \big]\,\d t.$$
Substituting these equations into \eqref{3-prop-2.1} and by the definition of $g^\delta$, we arrive at
  \begin{equation} \label{3-prop-2.2}
  g^\delta\big( X^\delta_t(X_0) \big) = g^\delta(X_0) - \sum_{1\leq i\neq j \leq N} \big[ I^{ij}_1(t) + I^{ij}_2(t) +I^{ij}_3(t) +I^{ij}_4(t) \big],
  \end{equation}
where
  $$\aligned
  I^{ij}_1(t)=&\, \sum_{k\neq i} \xi_k \int_0^t \nabla G^\delta \big(x^\delta_i(t)- x^\delta_j(t)\big) \cdot K^\delta \big(x^\delta_i(t) -x^\delta_k(t) \big)\,\d t, \\
  I^{ij}_2(t)=&\, - \sum_{k\neq j} \xi_k \int_0^t \nabla G^\delta \big(x^\delta_i(t)- x^\delta_j(t)\big) \cdot K^\delta \big(x^\delta_j(t) -x^\delta_k(t) \big)\,\d t, \\
  I^{ij}_3(t)=&\, \sum_{l=1}^\infty \int_0^t \nabla G^\delta \big(x^\delta_i(t)- x^\delta_j(t)\big) \cdot \big[\sigma_l \big(x^\delta_i(t) \big) - \sigma_l \big(x^\delta_j(t) \big) \big]\, \d W^l_t, \\
  I^{ij}_4(t)=&\, \frac12 \sum_{\alpha,\beta=1}^2 \sum_{l=1}^\infty \int_0^t \partial_{\alpha,\beta} G^\delta \big(x^\delta_i(t)- x^\delta_j(t)\big) \big[\sigma_l^\alpha \big(x^\delta_i(t) \big) - \sigma_l^\alpha \big(x^\delta_j(t) \big) \big] \\
  &\hskip70pt \times \big[\sigma_l^\beta \big(x^\delta_i(t) \big) - \sigma_l^\beta \big(x^\delta_j(t) \big) \big]\,\d t.
  \endaligned$$

We estimate the four terms one by one. First, by the definition \eqref{regular-kernel} of the kernel $K^\delta$,
  $$\nabla G^\delta \big(x^\delta_i(t)- x^\delta_j(t)\big) \cdot K^\delta \big(x^\delta_i(t) -x^\delta_j(t) \big) =0,$$
which gives us the key cancellation since this term is the most singular one: by the properties of $G^\delta$, it behaves like
  $$\frac1{\big| x^\delta_i(t) -x^\delta_j(t) \big|^{4-2\eps}}$$
for small $\big| x^\delta_i(t) -x^\delta_j(t) \big|$; while the other terms behave like
  $$\frac1{\big| x^\delta_i(t) -x^\delta_j(t) \big|^{2-\eps}} \frac1{\big| x^\delta_i(t) -x^\delta_k(t) \big|^{2-\eps}}$$
with $j\neq k$. Hence we obtain
  $$\big|I^{ij}_1(t) \big| \leq C \sum_{k\neq i,j} \int_0^t \frac1{\big| x^\delta_i(s) -x^\delta_j(s) \big|^{2-\eps}} \frac1{\big| x^\delta_i(s) -x^\delta_k(s) \big|^{2-\eps}} \,\d s. $$
In the same way,
  $$\big|I^{ij}_2(t) \big| \leq C \sum_{k\neq i,j} \int_0^t \frac1{\big| x^\delta_i(s) -x^\delta_j(s) \big|^{2-\eps}} \frac1{\big| x^\delta_j(s) -x^\delta_k(s) \big|^{2-\eps}} \,\d s. $$
Then by the definition of $h_1(X)$,
  \begin{equation}\label{3-prop-2.3}
  \sum_{1\leq i\neq j \leq N} \big( \big|I^{ij}_1(t) \big| + \big|I^{ij}_2(t) \big|\big) \leq C\int_0^t h_1\big( X^\delta_s(X_0) \big)\,\d s.
  \end{equation}

Next, recalling the definition of $I^{ij}_3(t)$ and by Burkholder--Davis--Gundy's inequality,
  $$\aligned
  \E\bigg[ \sup_{t\in [0,T]} \big|I^{ij}_3(t)\big| \bigg] & \leq C \E\Big(\big[I^{ij}_3, I^{ij}_3\big]_T^{1/2} \Big) \\
  &\leq C \bigg[\sum_{l=1}^\infty \E \int_0^T \Big(\nabla G^\delta \big(x^\delta_i(t)- x^\delta_j(t)\big) \cdot \big[\sigma_l \big(x^\delta_i(t) \big) - \sigma_l \big(x^\delta_j(t) \big) \big] \Big)^2 \,\d t \bigg]^{1/2}.
  \endaligned$$
Using the Lipschitz estimate of $\{\sigma_l\}_{l\geq 1}$ (see Remark \ref{1-rem}(a)) and the regularity properties \eqref{property-regular-Green} of $G^\delta$, we get
  $$\E\bigg[ \sup_{t\in [0,T]} \big|I^{ij}_3(t)\big| \bigg] \leq C \bigg[ \E \int_0^T \frac1{\big| x^\delta_i(s) -x^\delta_j(s) \big|^{2-2\eps}} \,\d s \bigg]^{1/2},$$
Therefore,
  \begin{equation}\label{3-prop-2.4}
  \sum_{1\leq i\neq j \leq N} \E \bigg[ \sup_{t\in [0,T]} \big|I^{ij}_3(t)\big| \bigg] \leq C_N  \bigg[ \E \int_0^T h_2\big( X^\delta_s(X_0) \big) \,\d s \bigg]^{1/2}.
  \end{equation}

Finally, similarly to the arguments in the last step,
  $$\aligned
  \big|I^{ij}_4(t)\big| &\leq C\sum_{\alpha,\beta=1}^2 \int_0^t \big| \partial_{\alpha,\beta} G^\delta \big(x^\delta_i(s)- x^\delta_j(s)\big) \big| \cdot \big| x^\delta_i(s) -x^\delta_j(s) \big|^2\,\d s \\
  &\leq C \int_0^t \frac1{\big| x^\delta_i(s) -x^\delta_j(s) \big|^{1-\eps}}\,\d s,
  \endaligned $$
which implies
  $$\sum_{1\leq i\neq j \leq N} \big|I^{ij}_4(t)\big| \leq C \int_0^t h_1\big( X^\delta_s(X_0) \big)\,\d s.$$
Combining this estimate with \eqref{3-prop-2.2}--\eqref{3-prop-2.4}, we complete the proof.
\end{proof}

As a consequence of Proposition \ref{3-prop-2}, we have
  $$\aligned
  \int_{(\T^2)^N} \E\bigg[ \sup_{t\in [0,T]} g^\delta\big( X^\delta_t(X_0) \big) \bigg] \d X_0  \leq &\, \int_{(\T^2)^N} g^\delta(X_0)\,\d X_0 + C_1 \int_0^T \E\int_{(\T^2)^N} h_1\big( X^\delta_t(X_0) \big) \,\d X_0\d t \\
  &\, + C_2\bigg[\int_0^T \E\int_{(\T^2)^N} h_2\big( X^\delta_t(X_0) \big) \,\d X_0\d t \bigg]^{1/2},
  \endaligned $$
where we have used Cauchy's inequality for the last term, since $\d X_0$ is a probability measure on $(\T^2)^N$. Using Lemma \ref{3-lem-1} and integrability of the functions $g^\delta$, $h_1$ and $h_2$, we arrive at
  \begin{equation}\label{key-estimate}
  \aligned
  \int_{(\T^2)^N} \E\bigg[ \sup_{t\in [0,T]} g^\delta\big( X^\delta_t(X_0) \big) \bigg] \d X_0 \leq &\, \int_{(\T^2)^N} g^\delta(X_0)\,\d X_0 + C_1 \int_0^T \E\int_{(\T^2)^N} h_1( Y) \,\d Y\d t \\
  &\, + C_2\bigg[\int_0^T \E\int_{(\T^2)^N} h_2( Y) \,\d Y\d t \bigg]^{1/2} \\
  \leq &\ C<+\infty,
  \endaligned
  \end{equation}
where the constant $C$ depends only on $\eps\in (0,1)$ and $N\in \N$.

In the following we write $x^\delta_i(t|X_0)\ (1\leq i \leq N)$ for the components of $X^\delta_t(X_0)$ in order to emphasize the dependence on the initial configuration $X_0\in (\T^2)^N$. Now we can prove

\begin{corollary}\label{3-cor-3}
There is a constant $C>0$ such that for any $\delta\in (0,1)$, it holds
  $$\big( {\rm Leb}_{(\T^2)^N} \otimes \P \big)\bigg(\min_{1\leq i\neq j\leq N} \inf_{t\in [0,T]} \big| x^\delta_i(t|X_0) -x^\delta_j(t|X_0) \big| \leq \delta \bigg) \leq C \delta^{1-\eps} .$$
\end{corollary}

\begin{proof}
Note that, by definition,
  $$g^\delta\big( X^\delta_t(X_0) \big) = -\sum_{1\leq i\neq j\leq N} \big[ G^\delta \big(x^\delta_i(t|X_0) -x^\delta_j(t|X_0) \big) - c_0 \big].$$
Recall that $c_0\geq 0$ is such that $G^\delta(x) \leq c_0$ for all $x\in \T^2$. For $\delta>0$ small enough, if
  $$\min_{1\leq i\neq j\leq N} \inf_{t\in [0,T]} \big| x^\delta_i(t|X_0) -x^\delta_j(t|X_0) \big| \leq \delta,$$
then there exist $t_0\in [0,T]$ and $i_0,j_0$ such that $\big| x^\delta_{i_0}(t_0|X_0) -x^\delta_{j_0}(t_0|X_0) \big| \leq \delta$. Therefore, by the definition of $G^\delta$,
  $$g^\delta\big( X^\delta_{t_0}(X_0) \big) \geq - \big[ G^\delta \big( x^\delta_{i_0}(t_0|X_0) -x^\delta_{j_0}(t_0|X_0) \big) - c_0 \big] \geq c_0 + \frac{C}{\delta^{1-\eps}},$$
which implies
  $$\sup_{t\in [0,T]} g^\delta\big( X^\delta_t(X_0) \big) \geq \frac{C}{\delta^{1-\eps}}.$$
Combining this result with Chebyshev's inequality and \eqref{key-estimate}, we get
  \[\aligned
  &\ \big( {\rm Leb}_{(\T^2)^N} \otimes \P \big)\bigg(\min_{1\leq i\neq j\leq N} \inf_{t\in [0,T]} \big| x^\delta_i(t|X_0) -x^\delta_j(t|X_0) \big| \leq \delta \bigg) \\
  \leq &\ \big( {\rm Leb}_{(\T^2)^N} \otimes \P \big) \bigg( \sup_{t\in [0,T]} g^\delta\big( X^\delta_t(X_0) \big) \geq \frac{C}{\delta^{1-\eps}} \bigg) \\
  \leq &\ C^{-1} \delta^{1-\eps} \int_{(\T^2)^N} \E\bigg[ \sup_{t\in [0,T]} g^\delta\big( X^\delta_t(X_0) \big) \bigg] \d X_0 \leq C' \delta^{1-\eps}.
  \endaligned \]
The proof is complete.
\end{proof}

Now we follow the arguments at the end of \cite[Section 3]{FGP}. Recall the definition of the generalized diagonal $\Delta_N$ of $(\T^2)^N$. For any initial configuration $X_0\in \Delta_N^c$ (the complement of $\Delta_N$ in $(\T^2)^N$), the stochastic point vortex system \eqref{stoch-vortex-model} makes sense until the solution $X_t(X_0) = (x_1(t), \ldots, x_N(t))$ reaches $\Delta_N$. Let $\Delta_N^\delta$ be the $\delta$-neighborhood of $\Delta_N$ in $(\T^2)^N$. For $X_0\in \big(\Delta_N^\delta \big)^c$, define the stopping time
  $$\tau_{X_0}^\delta(\omega)= \inf\big\{t>0: X^\delta_t(X_0,\omega) \in \Delta_N^\delta \big\},$$
with the convention that $\inf\emptyset = +\infty$. The continuity of trajectories implies $\P \big( \tau_{X_0}^\delta>0 \big)=1$. Moreover, on the random interval $\big[0, \tau_{X_0}^\delta \big]$, the solution $X^\delta_t(X_0)$ coincides with the unique solution  $X_t$ of \eqref{stoch-vortex-model}, which implies $\tau_{X_0}^\delta$ is also the first instant that $X_t$ enter $\Delta_N^\delta$. We define
  \begin{equation}\label{stopping-time}
  \tau_{X_0} = \sup_{\delta\in (0,1)} \tau_{X_0}^\delta.
  \end{equation}
Then the unique solution $X_t$ of the system \eqref{stoch-vortex-model} is well defined on the interval $[0, \tau_{X_0} )$.

\begin{proposition}\label{3-prop-4}
For ${\rm Leb}_{(\T^2)^N}$-a.e. $X_0\in (\T^2)^N$, the stochastic point vortex system \eqref{stoch-vortex-model} has a globally defined unique strong solution.
\end{proposition}

\begin{proof}
It is sufficient to show that $\P(\tau_{X_0} = \infty)=1$ holds for ${\rm Leb}_{(\T^2)^N}$-a.e. $X_0\in (\T^2)^N$. To this end, for any given $T>0$ and $\delta_0\in (0,1)$, we prove that $\P(\tau_{X_0} \geq T)=1$ for ${\rm Leb}_{(\T^2)^N}$-a.e. $X_0\in \big(\Delta_N^{\delta_0} \big)^c$.

By Corollary \ref{3-cor-3}, for any $\delta\in (0, \delta_0)$,
  $$\big( {\rm Leb}_{(\T^2)^N} \otimes \P \big)\bigg(\min_{1\leq i\neq j\leq N} \inf_{t\in [0,T]} \big| x^\delta_i(t|X_0) -x^\delta_j(t|X_0) \big| \leq \delta \bigg) \leq C \delta^{1-\eps} .$$
Take any sequence $\{\delta_k\}_{k\geq 1} \subset (0, \delta_0)$ which tends to 0 fast enough such that $\sum_{k=1}^\infty \delta_k^{1-\eps} <+\infty$. The Borel--Cantelli lemma yields the existence a $\big( {\rm Leb}_{(\T^2)^N} \otimes \P \big)$-negligible set $A \subset (\T^2)^N \times \Omega$, with the property that for all $(X_0, \omega)\in A^c$ there exists $k_0 = k_0(X_0, \omega) \geq 1$ such that for all $k\geq k_0$, one has
  $$\min_{1\leq i\neq j\leq N} \inf_{t\in [0,T]} \big| x^{\delta_k}_i(t|X_0) -x^{\delta_k}_j(t|X_0) \big| > \delta_k. $$
When restricted to $(X_0, \omega)\in A^c \cap \big(\Delta_N^{\delta_0} \big)^c \times \Omega$, the above assertion implies $\tau^{\delta_k}_{X_0} (\omega) \geq T$ for all $k\geq k_0$. Hence, $\tau_{X_0}(\omega) \geq T$ by the definition \eqref{stopping-time}. Since $A$ is $\big( {\rm Leb}_{(\T^2)^N} \otimes \P \big)$-negligible, the inequality $\tau_{X_0}(\omega) \geq T$ holds for a.e. $(X_0, \omega)\in \big(\Delta_N^{\delta_0} \big)^c \times \Omega$. By the Fubini theorem, there is a full measure set $B\subset \big(\Delta_N^{\delta_0} \big)^c$, i.e. ${\rm Leb}_{(\T^2)^N} \big(\big( \Delta_N^{\delta_0} \big)^c \setminus B\big) = 0$, such that for all $X_0\in B$ it holds $\P(\tau_{X_0} \geq T\big)=1$, which completes the proof.
\end{proof}

With the above preparations, finally we can prove the main result of this paper by following the ideas in the proof of \cite[Theorem 1, p.1457]{FGP}.

\begin{proof}[Proof of Theorem \ref{main-result}]
For any given $X_0\in \Delta_N^c$, the system \eqref{stoch-vortex-model} has a unique strong local solution on the random interval $\big[0, \tau_{X_0} \big)$. By Proposition \ref{3-prop-4}, for ${\rm Leb}_{(\T^2)^N}$-a.e. $X_0\in \Delta_N^c$, $\P(\tau_{X_0}= +\infty)=1$. We want to show that this property is true for all $X_0\in \Delta_N^c$. Let us add a point $\Lambda$ to $(\T^2)^N$ and, when $\tau_{X_0}<+\infty$, we set $X_t= X_t(X_0) = \Lambda$ for $t\geq \tau_{X_0}$. The process $X_t$ lives on $\Delta_N^c \cup \{\Lambda \}$ and it is Markovian. We have, for $\eta>0$ small enough,
  $$\P\big( X_{[\eta,T]}(X_0) \subset \Delta_N^c \big) = \int_{\Delta_N^c \cup \{\Lambda \}} \P\big( X_{[0,T-\eta]}(Y) \subset \Delta_N^c \big)\, \mu_{X_\eta (X_0)}(\d Y),$$
where $\big\{ X_{[\eta,T]}(X_0) \subset \Delta_N^c \big\} = \big\{\omega\in \Omega: X_t(X_0,\omega) \in \Delta_N^c \mbox{ for any } t\in [\eta, T] \big\}$, and $\mu_{X_\eta (X_0)}$ is the law of $X_\eta (X_0)$. Let $A\subset (\T^2)^N$ be a ${\rm Leb}_{ (\T^2)^N}$-negligible set such that for all $X_0\in A^c$, the Cauchy problem for \eqref{stoch-vortex-model} has a unique global solution. We have $\P\big( X_{[0,T-\eta]}(Y) \subset \Delta_N^c \big) =1$ for all $Y\in A^c$. Thus,
  $$\P\big( X_{[\eta,T]}(X_0) \subset \Delta_N^c \big) \geq \int_{A^c} \P\big( X_{[0,T-\eta]}(Y) \subset \Delta_N^c \big)\, \mu_{X_\eta (X_0)}(\d Y) = 1- \mu_{X_\eta (X_0)}(A).$$

Now for small $\delta_0>0$, assume $X_0\in \big(\Delta_N^{\delta_0} \big)^c$. Then for all $\delta\in (0, \delta_0)$,
  $$\aligned
  \mu_{X_\eta (X_0)}(A) &= \P(X_\eta (X_0) \in A) \\
  &= \P\big( X_\eta (X_0) \in A, \tau^\delta_{X_0}> \eta \big) + \P\big( X_\eta (X_0) \in A, \tau^\delta_{X_0}\leq \eta \big)\\
  &\leq \P\big( X^\delta_\eta (X_0) \in A \big) + \P\big( \tau^\delta_{X_0}\leq \eta \big)\\
  &= \P\big( \tau^\delta_{X_0}\leq \eta \big),
  \endaligned$$
where in the last step we have used the facts that $A$ is ${\rm Leb}_{(\T^2)^N}$-negligible and that the law of $X^\delta_\eta (X_0)$ is absolutely continuous with respect to ${\rm Leb}_{(\T^2)^N}$ for any $X_0\in \Delta_N^c,\, \delta>0$ and $\eta >0$. The second assertion follows from the hypothesis \textbf{(H2)} and Corollary \ref{cor-density}.

Next, the continuity of trajectories leads to $\lim_{\eta\to 0} \P\big( \tau^\delta_{X_0}\leq \eta \big) =0$; therefore,
  $$\lim_{\eta\to 0} \P\big( X_{[\eta,T]}(X_0) \subset \Delta_N^c \big) =1.$$
Note that the sequence of events $\big\{ X_{[1/n,T]}(X_0) \subset \Delta_N^c\big\}$ is decreasing in $n$, so is the probability $\P\big( X_{[1/n,T]}(X_0) \subset \Delta_N^c\big)$. As a result, for any $\eta>0$, we have $\P\big( X_{[\eta,T]}(X_0) \subset \Delta_N^c \big) =1$ and hence $\P\big( X_{[0,T]}(X_0) \subset \Delta_N^c \big) =1$.
\end{proof}

\section{An explicit blow-up result}

In this section we show that,  for the deterministic system \eqref{vortex-model}, one cannot go further than proving existence for almost every initial condition, unless all the intensities have the same sign. Hence, the result on the global existence of the stochastic system \eqref{stoch-vortex-model} for every initial condition shown above is indeed an example of regularization by noise.

To prove the occurrence of collapse we follow here the idea of Marchioro and Pulvirenti \cite[Section 4.6]{MP} and give explicit examples of initial conditions for the three point vortex system leading to a collapse. The readers can find in \cite[Section IV]{BB} detailed discussions on the collapse of vortex models of the SQG equations, i.e. the case with $\eps=0$ in \eqref{vortex-model}. To avoid technicalities we will work in the whole plane, that is, we consider
\begin{equation*}
  \frac{\d x_i(t)}{\d t}=  \sum_{j=1,j\neq i}^3 \xi_j K_\eps (x_i(t) -x_j(t)),\quad i=1,2,3,
  \end{equation*}
where $\xi_j\in \R\setminus \{0\}$ is the intensity of the vortex point $x_j(t)$, $j=1,2,3$, and
\begin{align*}
 K_\eps: \R^2\setminus\{0\}\to \R^2, \quad K_\eps(x)= c_\eps \frac{x^\perp}{\left|x\right|^{3-\eps}}.
\end{align*}

Let us consider the distance $l_{i,j}(t)=\left|x_i(t)- x_j(t)\right|$ of two different vortices. Then,
\begin{align*}
 \frac{\d}{\d t} l_{i,j}(t)^2 & =2 (x_i(t)-x_j(t))\cdot \left(\frac{\d x_i(t)}{\d t}-\frac{\d x_j(t)}{\d t} \right) \\
			       & = 2 (x_i(t)-x_j(t))\cdot \left(\sum_{k=1,k\neq i}^3 \xi_k K_\eps (x_i(t) -x_k(t))-\sum_{m=1,m\neq j}^3 \xi_m K_\eps (x_j(t) -x_m(t))\right)
\end{align*}
Using
\begin{align*}
(x_i(t)-x_j(t))\cdot \big( \xi_j K_\eps (x_i(t) -x_j(t))-\xi_i K_\eps (x_j(t) -x_i(t)) \big)=0
\end{align*}
we obtain, for $i\neq k\neq j$,
\begin{align*}
 \frac{\d}{\d t} l_{i,j}(t)^2 &= 2\xi_k (x_i(t)-x_j(t))\cdot \left(K_\eps(x_i(t) -x_k(t))- K_\eps (x_j(t) -x_k(t))\right) \\
 & =2c_\eps \xi_k (x_i(t)-x_j(t))\cdot \left(\frac{(x_i(t) -x_k(t))^\perp}{\left|x_i(t) -x_k(t)\right|^{3-\eps}}-\frac{(x_j(t) -x_k(t))^\perp}{\left|x_j(t) -x_k(t)\right|^{3-\eps}}\right) \\
 & = 2c_\eps \xi_k \left(\frac{-x_i(t)\cdot x_k(t)^\perp -x_j(t)\cdot x_i(t)^\perp + x_j(t)\cdot x_k(t)^{\perp}}{\left|x_i(t) -x_k(t)\right|^{3-\eps}}\right. \\
 &\qquad\qquad \qquad \left.-\frac{ x_i(t)\cdot x_j(t)^\perp -x_i(t)\cdot x_k(t)^\perp + x_j(t)\cdot x_k(t)^{\perp}}{\left|x_j(t) -x_k(t)\right|^{3-\eps}}\right).
\end{align*}
Note that $x_i(t)\cdot x_j(t)^\perp= -x_j(t)\cdot x_i(t)^\perp$, we obtain
\begin{align*}
 \frac{\d}{\d t} l_{i,j}(t)^2 &= 2c_\eps \xi_k \left(\frac{x_j(t)\cdot x_k(t)^{\perp}- x_i(t)\cdot x_k(t)^\perp -x_j(t)\cdot x_i(t)^\perp} {\left|x_i(t) -x_k(t)\right|^{3-\eps}}\right. \\
 &\qquad\qquad \qquad \left. -\frac{ x_j(t)\cdot x_k(t)^{\perp}-x_i(t)\cdot x_k(t)^\perp-x_j(t)\cdot x_i(t)^\perp}{\left|x_j(t) -x_k(t)\right|^{3-\eps}}\right).
\end{align*}
The area of the triangle spanned by the vectors $x_j(t)$ and $x_i(t)$ is given by $\frac12 |x_j(t)\cdot x_i(t)^\perp|$. Hence, $A(t):= x_j(t)\cdot x_k(t)^{\perp} -x_i(t)\cdot x_k(t)^\perp -x_j(t)\cdot x_i(t)^\perp$ is twice of the area of the triangle with endpoints $x_j(t)$, $x_i(t)$ and $x_k(t)$. The value of $A(t)$ is positive if and only if $x_j(t)$, $x_i(t)$ and $x_k(t)$ are ordered counter-clockwise. It follows
\begin{align*}
 \frac{\d}{\d t} l_{i,j}(t)^2= 2c_\eps \xi_k A(t) \left(\frac{1}{l_{i,k}(t)^{3-\eps}}-\frac{1}{l_{k,j}(t)^{3-\eps}}\right).
\end{align*}
Thus,
\begin{align*}
S:=\frac{ l_{1,2}(t)^2}{\xi_3 }+\frac{ l_{2,3}(t)^2}{\xi_1 }+\frac{ l_{3,1}(t)^2}{\xi_2 }
\end{align*}
is constant in time.

To obtain another invariant, we consider
\begin{align*}
\frac{\d}{\d t} l_{i,j}(t)^{-1+\eps}
& =(-1+\eps)l_{i,j}(t)^{-2+\eps}  \frac{\d}{\d t} l_{i,j}(t)= \frac12 (-1+\eps)l_{i,j}(t)^{-3+\eps} \frac{\d}{\d t} l_{i,j}(t)^2\\
&= (-1+\eps) c_\eps \xi_k A(t) \frac{1}{l_{i,j}(t)^{3-\eps}}\left(\frac{1}{l_{i,k}(t)^{3-\eps}}-\frac{1}{l_{k,j}(t)^{3-\eps}}\right).
\end{align*}
Hence,
\begin{align*}
 S_{\eps}:=\frac{ l_{1,2}(t)^{-1+\eps}}{\xi_3 }+\frac{ l_{2,3}(t)^{-1+\eps}}{\xi_1 }+\frac{ l_{3,1}(t)^{-1+\eps}}{\xi_2 }
\end{align*}
is also constant in time. We get
\begin{align*}
 \frac{\d}{\d t}\left( \frac{ l_{1,2}(t)^2 }{ l_{2,3}(t)^2 }  \right)
 & =\frac{ 2c_\eps A(t) }{l_{2,3}(t)^4 }
 \left( l_{2,3}(t)^2 \xi_3 \left(\frac{1}{l_{3,1}(t)^{3-\eps}}-\frac{1}{l_{2,3}(t)^{3-\eps}}\right)\right.\\
 &\qquad\qquad\qquad\qquad  \left.-l_{1,2}(t)^2\xi_1  \left(\frac{1}{l_{1,2}(t)^{3-\eps}}-\frac{1}{l_{3,1}(t)^{3-\eps}}\right) \right)\\
 &=\frac{ 2c_\eps A(t) }{l_{2,3}(t)^4 }
 \left( \frac{l_{2,3}(t)^2 \xi_3+l_{1,2}(t)^2\xi_1}{l_{3,1}(t)^{3-\eps}}-\frac{\xi_3 }{l_{2,3}(t)^{1-\eps}}-\frac{\xi_1}{l_{1,2}(t)^{1-\eps}}\right).
\end{align*}
From the definition of $S$, we have
\begin{equation}\label{eq-ratio}
\aligned
 \frac{\d}{\d t}\left( \frac{ l_{1,2}(t)^2 }{ l_{2,3}(t)^2 }  \right)  & =\frac{ 2c_\eps A(t) }{l_{2,3}(t)^4 }
 \left( \frac{\xi_1\xi_3 S}{l_{3,1}(t)^{3-\eps}}-\frac{\xi_3 }{l_{2,3}(t)^{1-\eps}}-\frac{\xi_1}{l_{1,2}(t)^{1-\eps}}-\frac{\xi_3\xi_1}{\xi_2}\frac{1}{l_{3,1}(t)^{1-\eps}}\right)\\
 & =\xi_1\xi_3 \frac{ 2c_\eps A(t) }{l_{2,3}(t)^4 }
 \left(S\, l_{3,1}(t)^{-3+\eps} -S_\eps \right) .
\endaligned
\end{equation}
Choosing initial positions such that $S=0$ (which is also a necessary condition for a collapse of the three vortices) and $S_\eps=0$, we get that the ratio  $\frac{ l_{1,2}(t)^2 }{ l_{2,3}(t)^2 }$ is  constant in time and similarly we deduce that $\frac{ l_{2,3}(t)^2 }{ l_{3,1}(t)^2 }$ and $\frac{ l_{3,1}(t)^2 }{ l_{1,2}(t)^2 }$ are constant in time. Note that $S_\eps=0$ is a condition on the intensities only in the case of the Euler equation, i.e., for $\eps=1$. Hence, the triangle $A(t)$ does not change its shape in time and $\frac{A(t)}{l_{i,j}(t)^2}$ is also independent of $t$. We can give its value in terms of the angle $\varphi_{i,j}$ between the vectors $x_j-x_k$ and $x_i-x_k$:
\begin{align*}
\frac{A(t)}{l_{i,j}(t)^2}=\frac12 \frac{1}{l_{i,j}(t)^2} l_{i,k}(t)l_{k,j}(t)\sin(\varphi_{i,j}) =\frac12 \frac{l_{i,k}(0)l_{k,j}(0)}{l_{i,j}(0)^2}\sin(\varphi_{i,j}).
\end{align*}
This yields now, with $l_{i,k}(t)=\frac{l_{i,k}(0)}{l_{i,j}(0)}l_{i,j}(t)$,
\begin{align*}
 \frac{\d}{\d t} l_{i,j}(t)^2
 & = 2c_\eps \xi_k A(t) \left(\frac{1}{l_{i,k}(t)^{3-\eps}}-\frac{1}{l_{k,j}(t)^{3-\eps}}\right)\\
 & = 2c_\eps \xi_k A(t) \left(\frac{l_{i,j}(0)^{3-\eps}}{l_{i,k}(0)^{3-\eps}}
 -\frac{l_{i,j}(0)^{3-\eps}}{l_{k,j}(0)^{3-\eps}}\right)\frac{1}{l_{i,j}(t)^{3-\eps}}\\
  & = c_\eps \xi_k l_{i,k}(0)l_{k,j}(0) l_{i,j}(0)^{1-\eps}\sin(\varphi_{i,j}) \left(\frac{1}{l_{i,k}(0)^{3-\eps}}
 -\frac{1}{l_{k,j}(0)^{3-\eps}}\right)\frac{1}{l_{i,j}(t)^{1-\eps}}\\
 &=: c_{i,j} \frac{1}{l_{i,j}(t)^{1-\eps}} .
\end{align*}
We can solve this ODE explicitly,
\begin{align*}
  l_{i,j}(t)=\left( l_{i,j}(0)^{(3-\eps)/2}+\frac12 c_{i,j} (3-\eps) t \right)^{1/(3-\eps)}.
\end{align*}
Therefore, we obtain a collapse if we have $S=0=S_{\varepsilon}$ and $c_{i,j} <0$.

That there are choices for the initial values and intensities under which all the conditions are satisfied is easy to see. We first take initial values which form a triangle with $l_{1,2}(0) >l_{2,3}(0)$ and $\xi_2=1$, then $c_{1,3}<0$ and we can choose $\xi_1$ and $\xi_3$ such that $S=0=S_{\varepsilon}$. For example, we can consider initial positions given in \cite[Section 4.2, p.140]{MP}:
  $$x_1(0)=(-1,0), \quad x_2(0)=(1,0), \quad x_3(0)=(1,\sqrt{2}\, );$$
then
  $$l_{1,2}(0)=2,\quad l_{2,3}(0)=\sqrt{2}, \quad l_{3,1}(0) = \sqrt{6}. $$
Taking $\xi_2=1$, then we deduce from $S=0=S_{\varepsilon}$ the following equations
\begin{align*}
 0=\frac{ 4}{\xi_3 }+\frac{2}{\xi_1 }+ 6
\end{align*}
and
\begin{align*}
0=\frac{2^{-1+\eps}}{\xi_3 }+\frac{2^{(-1+\eps)/2 }}{\xi_1 }+ 6^{(-1+\eps)/2 }.
\end{align*}
Choose $\eps =1/2$, the last equation becomes
  $$0=\frac{2^{-1/2}}{\xi_3 }+\frac{2^{-1/4}}{\xi_1 }+ 6^{-1/4}.$$
Solving these equations gives us
  $$\xi_1= 1.155616, \quad \xi_3= -0.517419. $$
With these data in hand, we can do a simulation and obtain a similar figure as \cite[Figure 4.4]{MP}.

In the above computations we have assumed $S_{\eps}=0$ to determine the intensities $\xi_1$ and $\xi_3$. This condition is sufficient for our purpose but it is not necessary (on the contrary $S$ must be 0). However, in view of \eqref{eq-ratio}, if $S_{\eps}\neq 0$, then the collapse of three point vortices will not be self-similar, see \cite[Section IV]{BB} for related discussions.

By a scaling argument this examples yield also a collapse in the torus. Let $x_i^{\lambda}(t):= \lambda^{\alpha}x_i(\lambda t)$ for $\lambda>0$, $\alpha\neq 0$ and $i=1,2,3$. Then
\begin{align*}
  \frac{\d}{\d t} x_i^{\lambda}(t)&= \lambda^{\alpha+1} \frac{\d x_i}{\d t}(\lambda t)
  =\lambda^{\alpha+1} \sum_{j=1,j\neq i}^3 \xi_j K_\eps (x_i(\lambda t) -x_j(\lambda t))\\
  &=\lambda^{\alpha+1}\frac{\lambda^{-\alpha}}{\lambda^{-\alpha(3-\eps)}} \sum_{j=1,j\neq i}^3 \xi_j K_\eps (\lambda^{\alpha} x_i(\lambda t) -\lambda^{\alpha} x_j(\lambda t))\\
    &=\frac{\lambda}{\lambda^{-\alpha(3-\eps)}} \sum_{j=1,j\neq i}^3 \xi_j K_\eps (x_i^{\lambda}(t) -x_j^{\lambda}(t)).
  \end{align*}
Hence, for $\alpha=-\frac{1}{3-\eps}$ the scaled point vortices $x_i^{\lambda}$ are again a solution to the initial data $x_i^{\lambda}(0):= \lambda^{\alpha}x_i(0)$. By taking now $\lambda$ sufficiently large we otain a collapse that occurs already in $[-1/2,1/2]^2$, without leaving the torus.

\end{document}